\documentclass[11pt]{article}

\usepackage[numbers]{natbib}
\usepackage{amsmath}
\usepackage{amssymb}
\usepackage{amsthm}
\usepackage{xcolor}
\usepackage{array}
\usepackage{booktabs}
\usepackage{tabularx}
\usepackage{hyperref}
\usepackage{algorithm}
\usepackage{float}
\usepackage[noend]{algpseudocode}
\usepackage{pgfplots}
\usepackage{placeins}
\pgfplotsset{compat=1.18}
\usepackage{enumitem}\usepackage[margin=1in]{geometry}

\hypersetup{colorlinks=true,citecolor=blue,linkcolor=blue,hypertexnames=false}
\allowdisplaybreaks
\usepackage[T1]{fontenc} 
\usepackage{fourier}

\DeclareMathAlphabet{\pazocal}{OMS}{zplm}{m}{n}
\SetMathAlphabet{\pazocal}{bold}{OMS}{zplm}{b}{n}
\usepackage{bm}
\newcommand{\paperTitle}{Three Standard Deviations Suffice While One Does Not}
\newcommand{\paperAuthor}{Victor Reis\thanks{Microsoft Research, Redmond. \texttt{victorol@microsoft.com}.} \and Zhao Song\thanks{\texttt{magic.linuxkde@gmail.com}.}}

\newcommand{\E}{\mathbb E}
\newcommand{\R}{\mathbb R}

\newcommand{\disc}{\operatorname{disc}}
\newcommand{\dist}{\operatorname{dist}}
\newcommand{\Pp}{\mathbb P}
\theoremstyle{plain}
\newtheorem{theorem}{Theorem}[section]
\newtheorem{lemma}[theorem]{Lemma}

\newtheorem{proposition}[theorem]{Proposition}

\theoremstyle{definition}

\begin{document}

\date{}
\title{\paperTitle}
\author{\paperAuthor}
\maketitle
\begin{abstract}
Spencer's 1985 ``six standard deviations suffice'' theorem shows that every $A \in [-1,1]^{n \times n}$ has a sign vector $x \in \{-1,1\}^n$ with $\|Ax\|_\infty \le 6\sqrt{n}$. We show the upper bound $\sqrt{3\operatorname{arsinh}(10)}\sqrt{n}+4 < 2.9992 \sqrt{n} + 4$ by directly rounding the minimizer of a potential function to a vertex of the cube.

We also show that, for every power of two $n \ge 2^{50}$, there exists a matrix $A \in \{-1,1\}^{n \times n}$ such that $\|Ax\|_\infty >1.0000002\sqrt{n}$ for every choice of signs $x \in \{-1,1\}^n$. The construction simply replaces a $2^{-22}$ fraction of the columns of a Hadamard matrix with independent random sign vectors. This is the first improvement over the $\sqrt{n}$ lower bound of Olson and Spencer (1978), which uses a Hadamard matrix. 
\end{abstract}


\section{Introduction}
Given a matrix $A\in\R^{d\times n}$, its \emph{discrepancy} is defined as $\disc(A):=\min_{x\in\{-1,1\}^n}\|Ax\|_\infty$. The square case $d = n$ is of particular interest. We define
\[
D_n^{\pm}:=\max_{A\in\{-1,1\}^{n\times n}}\frac{\disc(A)}{\sqrt n},
\qquad
D_n:=\max_{A\in[-1,1]^{n\times n}}\frac{\disc(A)}{\sqrt n}.
\]
Clearly $D_n^{\pm}\le D_n$. We first discuss lower bounds. A Hadamard matrix of order $n$ is a matrix $H\in\{-1,1\}^{n\times n}$ satisfying $H^\top H=nI_n$. This identity gives, for every signing $x\in\{-1,1\}^n$,
\[
\|Hx\|_\infty
\ge\frac{\|Hx\|_2}{\sqrt n}
=\|x\|_2
=\sqrt n.
\]
Taking the minimum over signings shows that $D^{\pm}_n \ge \disc(H)/\sqrt{n} \ge 1$. This was first observed by Olson and Spencer~\cite{OlsonSpencer}. We briefly mention the Hadamard conjecture, formulated by Paley in 1933~\cite{Paley}, which asserts that a Hadamard matrix of order $n$ exists for every positive integer $n$ divisible by four~\cite{deLauney}.

For the Sylvester Hadamard matrices $H_{2^m}:=H_2^{\otimes m}$ of order $n=2^m$, where
$H_2:=\begin{pmatrix}1&1\\1&-1\end{pmatrix}$, this lower bound is attained whenever the exponent $m$ is even. Indeed, for $y:=(1,1,1,-1)^\top$ we have $H_4y=2y$. Thus, if $m=2k$ with $k\ge1$, the signing $x:=y^{\otimes k}$ satisfies $H_{2^{2k}}x=(H_4 y)^{\otimes k}=2^k x$ and $\disc(H_{2^{2k}})=2^k=\sqrt n$.

On the other hand, $\disc(H_2)=2$, so $D^{\pm}_2\ge\sqrt2$. The best known lower bound for $D_n$ for some $n$ is obtained by replacing the last column of $H_{2^3}$ by zero:
\[
\widetilde{H}_{2^3}:=\begin{pmatrix}
1& 1& 1& 1& 1& 1& 1&0\\
1&-1& 1&-1& 1&-1& 1&0\\
1& 1&-1&-1& 1& 1&-1&0\\
1&-1&-1& 1& 1&-1&-1&0\\
1& 1& 1& 1&-1&-1&-1&0\\
1&-1& 1&-1&-1& 1&-1&0\\
1& 1&-1&-1&-1&-1& 1&0\\
1&-1&-1& 1&-1& 1& 1&0
\end{pmatrix}.
\]
This matrix has discrepancy $\disc(\widetilde{H}_{2^3})=5$, and therefore $D_8\ge5/\sqrt8>1.76$~\cite{ReisThesis}. In his 2024 open-problem list, Bandeira conjectured~\cite[Conjecture 12]{BandeiraEtAl} that
\[
\limsup_{n\to\infty}D_n^{\pm}>1.
\]
He also conjectured that $\disc(H_{2^m})=\sqrt{2}\,\sqrt{2^m}$ for odd $m$. The subsequent collected version credits Buhai with disproving this latter conjecture~\cite[Conjecture 13 and Updates]{BandeiraEtAl}. It already fails at order $512$, where~\cite[\S3.1]{KavutYucel} $\disc(H_{512})\le28<32= \sqrt{2} \sqrt{2^9}$. More strongly, Schmidt's Theorem 1~\cite{Schmidt} implies that
\[
\lim_{m\to\infty}\frac{\disc(H_{2^m})}{\sqrt{2^m}}=1.
\]
Partially inspired by the matrix $\widetilde{H}_{2^3}$, our first main result is a proof of Bandeira's conjecture by replacing a small fraction of the columns of a Hadamard matrix with random signs. We fix $\delta:=2^{-22}$ throughout.

\begin{theorem}\label{thm:main}
Let $n\ge 2^{50}$ be a power of two and $H_0\in\{-1,1\}^{n\times n}$ be a Hadamard matrix. Fix any set of $\delta n$ columns and let $A:=\bigl[\,H\mid R\,\bigr]$ be obtained by replacing those columns with a matrix of independent random signs $R$.  Then
\[
\Pp\bigl[\disc(A)>(1+\delta)\sqrt n\bigr]
\ge 1-2^{-n/20000}.
\]
In particular, $\limsup_{n\to\infty}D_n^{\pm} > 1.0000002$.
\end{theorem}

The proof is given in Section~\ref{sec:lower_bound}. Moving on to upper bounds, Spencer's ``six standard deviations suffice'' theorem~\cite{Spencer} gives $D_n\le6$; in fact it shows $D_n \le 5.32$~\cite[\S13.2]{AlonSpencer}. More generally, for $A \in [-1,1]^{d \times n}$, Spencer showed $\disc(A) \le 11 \sqrt{n\log(2d/n)}$. Bansal~\cite{Bansal} made the $O(\sqrt n)$ bound constructive in polynomial time, and Lovett and Meka~\cite{LovettMeka} gave a sticky random walk algorithm achieving the $O(\sqrt{n\log(2d/n)})$ bound. For the square case, Belshaw~\cite[Chapter 5]{Belshaw} obtained $D_n \le 5.2$ and claimed a proof of $D_n \le 3.7$ relying on private communications. Recently, Pesenti and Vladu~\cite[Theorem 4.5]{PesentiVladu} gave a polynomial-time algorithm achieving $4.1\sqrt n+O(1)$. Spielman conjectured that $D_n\le2$ for every $n\in \mathbb{N}$~\cite{Spielman}.

Our second main result gives the following bound for any number of vectors and any dimension; it is proved in Section~\ref{sec:direct_potential_upper_bound}. Recall that $\sinh t:=\frac{e^t-e^{-t}}2$ and $\operatorname{arsinh}t:=\log\bigl(t+\sqrt{1+t^2}\bigr)$ for $t \in \mathbb{R}$ is its inverse.
\begin{theorem}\label{thm:upper_intro}
For all $n,d \in \mathbb{N}$ and every matrix $A\in[-1,1]^{d\times n}$, there exists $x\in\{-1,1\}^n$ such that
\[
\|Ax\|_\infty\leq\sqrt{3n\operatorname{arsinh}(10d/n)}+4.
\]
In particular, $D_n\leq\sqrt{3\operatorname{arsinh}(10)}+4/\sqrt n$, which is less than 3 for $n\geq10^8$.
\end{theorem}

\subsection{Proof overview}

We give a high-level description of the proofs of Theorems~\ref{thm:main} and \ref{thm:upper_intro}.

\textbf{Lower bound.}
By orthogonality, for every sign vector $u \in \{-1,1\}^{(1-\delta)n}$, the vector $Hu$ has the same squared Euclidean norm $\|Hu\|_2^2 = (1-\delta)n^2$. The $\ell_1$ norm of $Hu$, however, varies. If all coordinates of $Hu$ were equal in magnitude, then $\|Hu\|_1$ would attain its maximal value, namely $\sqrt{1-\delta}\,n^{3/2}$. If instead $\|Hu\|_1 \le 0.99 n^{3/2}$, the coordinates whose magnitudes exceed $\sqrt n$ have large Euclidean norm. To produce small discrepancy, the random columns must compensate for this, which turns out to be very unlikely.

The harder case is when $\|Hu\|_1 > 0.99 n^{3/2}$. Talagrand's concentration inequality shows this happens with probability less than $2^{-n/310}$. For each fixed sign vector $x = (u,v) \in \{-1,1\}^n$, we show the probability that every coordinate of $Ax = Hu + Rv$ is less than $(1+\delta)\sqrt{n}$ is upper bounded by $2^{-n}2^{n/320}$. Summing this bound over all $2^n$ sign vectors $x$ would give $2^{n/320}$, but restricting to the exceptional fraction gives the exponentially small probability $2^{n/320}2^{-n/310}$. Combining the two cases shows that, with high probability, no sign vector achieves small discrepancy.

\textbf{Upper bound.}
For the upper bound, fix $A\in[-1,1]^{d\times n}$ and assume $n>4$. We define $R(z):=n-\|z\|_2^2$ for $z\in[-1,1]^n$. Rounding $z$ to its nearest sign vector $x$ changes $Az$ by at most $R(z)$ in the $\ell_\infty$ norm. Thus it suffices to find $z$ with $R(z)=4$ and $\|Az\|_\infty\leq \Delta:=\sqrt{3n\operatorname{arsinh}(10d/n)}$.

Without loss of generality, append the negation of every row to $A$ and continue to denote the resulting matrix by $A$. Write its rows as $a_1,\ldots,a_{2d}$. Thus $\|Az\|_\infty=\max_{1\leq i\leq2d}\langle a_i,z\rangle$, so it suffices to bound these inner products without absolute values. Set $s(z):=R(z)-4$ and define $\phi(t):=5/\sinh(t^2/3)$ for $t>0$ and $\phi(0):=\infty$, with the convention $\phi(\infty):=0$. We minimize the potential function
\[
\Phi(z):=\begin{cases}
\displaystyle\frac1{s(z)}\sum_{i=1}^{2d}\phi\Bigl(\frac{\Delta-\langle a_i,z\rangle}{\sqrt{s(z)}}\Bigr),&s(z)>0,\\[6pt]
0,&s(z)=0.
\end{cases}
\]
on the compact set $\pazocal K:=\{z\in[-1,1]^n:R(z)\geq4,\ \langle a_i,z\rangle\leq\Delta\text{ for every }i\in[2d]\}$. The potential is lower semicontinuous on $\pazocal K$, and the origin has $\Phi(0)<1$, so a finite minimum is attained. 

The barrier $\phi$ has two useful behaviors: $\phi(t)\sim15/t^2$ near zero, so approaching a row threshold incurs an unbounded penalty, while $\phi(t)\sim10e^{-t^2/3}$ for large $t$, so distant rows contribute little. A pure exponential $5e^{-t^2/3}$ would have the same type of tail but be finite at zero.

We argue by contradiction to show that no minimizer $z$ of $\Phi$ on $\pazocal K$ can have $R(z) > 4$; otherwise, we find a direction on which the restriction of $\Phi$ cannot have nonnegative second derivative. This uses both properties of $\phi$ and the fact that entries of $A$ are bounded.

\section{Proof of the lower bound}{\label{sec:lower_bound}}

\subsection{Preliminaries}

We will need Talagrand's concentration inequality to control the tail of the $\ell_1$ norm $\|Hu\|_1$. Here distance is Euclidean: $\dist(x,K):=\inf_{y\in K}\|x-y\|_2$. We refer to the expositions {\cite[\S1, (2)]{Tao}; \cite[Lemma 2.14]{Rothvoss}}.

\begin{lemma}[Talagrand's concentration inequality]\label{lem:talagrand}
Let $\varepsilon$ be uniform on $\{-1,1\}^d$, and let $K\subseteq\R^d$ be closed and convex with $\Pp[\varepsilon\in K]>0$. Then
\[
\E\left[\exp\left(\frac{\dist(\varepsilon,K)^2}{16}\right)\right]
\le \frac{1}{\Pp[\varepsilon\in K]}.
\]
\end{lemma}

For $k \in \mathbb{N}$, we denote by $S_k$ the random variable distributed as the sum $\sum_{j=1}^k \varepsilon_j$ of $k$ independent uniform random signs. We will need tail bounds on $S_k$; a similar bound was proved by Massart~\cite[Lemma 1.2]{Massart}. For completeness, the following estimates are proved in Appendix~A.

\begin{lemma}\label{lem:binomial}
Let $k\ge 2$ be even. Then the following hold:
\begin{enumerate}
\item[(a)] We have $\E[|S_k|]\le0.8\sqrt k$.
\item[(b)] Hoeffding's lemma~\cite[(4.16)]{Hoeffding}: For every $t\in\R$, we have $\E[\exp(tS_k/\sqrt k)]\le e^{t^2/2}$.
\item[(c)] For every $t\in\R$, we have $\Pp[S_k\le t\sqrt k-1]\le \frac12\exp(\sqrt{2/\pi}\,t)$.
\item[(d)] Let $u\ge0$ and suppose in addition that $k\ge2^{24}$. Then
\[
\Pp\left[u+\sqrt{\frac\delta k}\,S_k\le 1+\delta\right]
\le \frac12\exp\left(
4\sqrt\delta+
\frac{1-u^2}{\sqrt{2\pi\delta}}
\right).
\]
\end{enumerate}
\end{lemma}

We will need tail bounds when adding binomial noise $S_{\delta n}$ to a vector of given Euclidean norm. They follow from Lemma~\ref{lem:binomial}(d) above, and their proof can be found in Appendix~B.

\begin{lemma}\label{lem:noise}
Let $n\ge 2^{50}$ be a power of two, and let $Y$ be a random vector with i.i.d. coordinates distributed as $S_{\delta n}$. Then for every $z\in\R^n$ with $\|z\|_2^2=(1-\delta)n^2$,
\begin{enumerate}
\item[(a)] We have $\Pp[\|z+Y\|_\infty\le (1+\delta)\sqrt n]\le 2^{-n}2^{n/320}$.
\item[(b)] If in addition $\|z\|_1\le 0.99n^{3/2}$, the upper bound may be improved to $2^{-3n}$.
\end{enumerate}
\end{lemma}

\subsection{Proof of Theorem~\ref{thm:main}}

\begin{proof}[Proof of Theorem~\ref{thm:main}]
Recall $A =\bigl[\,H\mid R\,\bigr]$, where $H$ contains the $(1-\delta)n$ unchanged columns of $H_0$ and $R$ has $\delta n$ columns of independent random signs.

For a sign vector $x=(u,v)\in\{-1,1\}^{(1-\delta)n} \times \{-1,1\}^{\delta n}$, let $E_x$ denote the (bad) event that $\|Ax\|_\infty \le (1+\delta) \sqrt{n}$. Note that $Ax = Hu + Rv$. We have $\|Hu\|_2^2 = u^\top H^\top Hu =(1-\delta)n^2$, and $Rv$ is distributed as $Y$ in Lemma~\ref{lem:noise}. Hence
\[
\Pp_R\bigl[E_x\bigr] = \Pp_R\bigl[\|Hu + Rv\|_\infty\le(1+\delta)\sqrt n\bigr] 
\le 2^{-n}2^{n/320},
\]
with the stronger bound of Lemma~\ref{lem:noise}(b), namely $\Pp_R\bigl[E_x\bigr]  \le 2^{-3n}$, in case $\|Hu\|_1\le0.99n^{3/2}$. We claim that in fact this holds for most $u$.

\textbf{Claim.} Let $\varepsilon \sim \{-1,1\}^{(1-\delta)n}$ uniformly at random.  Then $\Pp_\varepsilon\bigl[\|H\varepsilon\|_1>0.99n^{3/2}\bigr] < 2^{-n/310}$.

\emph{Proof of Claim.} Each coordinate of $H\varepsilon$ is distributed as $S_{(1-\delta)n}$. Lemma~\ref{lem:binomial}(a) therefore gives
\[
\E_\varepsilon[\|H\varepsilon\|_1]
=n\E[|S_{(1-\delta)n}|]
\le 0.8\sqrt{1-\delta}\,n^{3/2}.
\]

Let $K:=\{u\in\R^{(1-\delta)n}:\|Hu\|_1\le 0.8n^{3/2}\}$ be a closed and convex set. For every $u\in K$,
\[
\|H\varepsilon\|_1\le \|Hu\|_1+\|H(\varepsilon-u)\|_1\le 0.8n^{3/2}+\sqrt n\,\|H(\varepsilon-u)\|_2=0.8n^{3/2}+n\|\varepsilon-u\|_2,
\]
where we used the triangle inequality and Cauchy--Schwarz.
Thus $\|H\varepsilon\|_1 \le 0.8n^{3/2}+n\dist(\varepsilon,K)$.

Using Markov's inequality and Lemma~\ref{lem:talagrand}, we obtain
\[
\Pp_\varepsilon\bigl[\|H\varepsilon\|_1>0.99n^{3/2}\bigr]
\le \Pp_\varepsilon\bigl[\dist(\varepsilon,K)>0.19\sqrt n\bigr]
\le \frac{\E_\varepsilon\left[\exp\left(\frac{\dist(\varepsilon,K)^2}{16}\right)\right]}{\exp\left(\frac{(0.19)^2n}{16}\right)}
\le \frac{\exp\left(-\frac{(0.19)^2n}{16}\right)}{\Pp_\varepsilon[\varepsilon\in K]}.
\]

Again by Markov's inequality,
\begin{align*}
\Pp[\varepsilon\in K]
&=1-\Pp\bigl[\|H\varepsilon\|_1>0.8n^{3/2}\bigr] \ge 1-\frac{\E_\varepsilon[\|H\varepsilon\|_1]}{0.8n^{3/2}} \ge 1-\sqrt{1-\delta}=\frac{\delta}{1+\sqrt{1-\delta}} \ge \frac\delta2.
\end{align*}

Thus $\Pp_\varepsilon\bigl[\|H\varepsilon\|_1>0.99n^{3/2}\bigr] \le \frac{2}{\delta} \exp\left(-\frac{(0.19)^2n}{16}\right) < 2^{-n/310}$, as claimed.

We conclude that
\[
\begin{aligned}
\Pp_R\bigl[\disc(A)\le(1+\delta)\sqrt n\bigr]
&=\Pp_R\bigl[\exists x\in\{-1,1\}^n:E_x\bigr]\\
&\le \sum_{x\in\{-1,1\}^n}\Pp_R[E_x]\\
&=\sum_{\substack{x=(u,v)\in\{-1,1\}^n\\ \|Hu\|_1\le0.99n^{3/2}}}\underbrace{\Pp_R[E_x]}_{\le 2^{-3n}}
+\sum_{\substack{x=(u,v)\in\{-1,1\}^n\\ \|Hu\|_1>0.99n^{3/2}}} \ \ \underbrace{\Pp_R[E_x]}_{\le 2^{-n} 2^{n/320}}\\
&\le 2^n2^{-3n}+\bigl(2^n2^{-n/310}\bigr)\bigl(2^{-n}2^{n/320}\bigr)\\
&<2^{-n/20000},
\end{aligned}
\]
where we used Lemma~\ref{lem:noise}(b) for the first sum and Lemma~\ref{lem:noise}(a) for the second, which has at most $2^n2^{-n/310}$ summands by the claim. Thus $\Pp\bigl[\disc(A)>(1+\delta)\sqrt n\bigr]
\ge 1-2^{-n/20000}$.
\end{proof}

\section{Proof of the upper bound}
\label{sec:direct_potential_upper_bound}

We restate the upper bound for convenience.

\begin{theorem}
\label{thm:direct_potential_upper_bound}
For all $n,d \in \mathbb{N}$ and every matrix $A\in[-1,1]^{d\times n}$, there exists $x\in\{-1,1\}^n$ such that
\[
\|Ax\|_\infty\leq\sqrt{3n\operatorname{arsinh}(10d/n)}+4.
\]
In particular, $D_n\leq\sqrt{3\operatorname{arsinh}(10)}+4/\sqrt n$, which is less than 3 for $n\geq10^8$.
\end{theorem}

For $z\in[-1,1]^n$, we define $R(z):=n-\|z\|_2^2$.

\begin{lemma}\label{lem:sinh_rounding}
Let $A\in[-1,1]^{d\times n}$ and $z\in[-1,1]^n$. Define its nearest-sign rounding $x\in\{-1,1\}^n$ by $x_j:=1$ when $z_j\geq0$, and $x_j:=-1$ otherwise. Then
\[
\|Ax\|_\infty\leq\|Az\|_\infty+R(z).
\]
\end{lemma}
\begin{proof}
Since $z_j^2\leq|z_j|$ for $z_j\in[-1,1]$,
\[
\|x-z\|_1=\sum_j(1-|z_j|)\leq\sum_j(1-z_j^2)=R(z).
\]
Every entry of $A$ has absolute value at most $1$, so $\|A(x-z)\|_\infty\leq\|x-z\|_1$. Hence
\[
\|Ax\|_\infty\leq\|Az\|_\infty+\|A(x-z)\|_\infty\leq\|Az\|_\infty+R(z).\qedhere
\]
\end{proof}

By Lemma~\ref{lem:sinh_rounding}, it suffices to find $z\in[-1,1]^n$ with $R(z)\leq4$ and
\begin{equation}
\|Az\|_\infty\leq \Delta:=\sqrt{3n\operatorname{arsinh}(10d/n)}.
\label{eq:sinh_target}
\end{equation}
For $n\leq4$, the point $z=0$ satisfies Eq.~\eqref{eq:sinh_target}. Henceforth assume $n>4$.

Without loss of generality, we may assume that each row of $A$ appears together with its negation: append the negation of every original row and retain the notation $A$ for the doubled $[-1,1]^{2d\times n}$ matrix. Thus $\|Az\|_\infty=\max_{i \in [2d]}\langle a_i,z\rangle$. Define $\phi:[0,\infty)\to\R\cup\{\infty\}$ by
\begin{equation}
\phi(t):=\begin{cases}
\displaystyle\frac5{\sinh(t^2/3)},&t>0,\\[4pt]
\infty,&t=0.
\end{cases}
\label{eq:sinh_profiles}
\end{equation}
We also use the limit convention $\phi(\infty):=0$. Assigning infinite potential at zero enforces a strictly positive row gap whenever $s(z)>0$. Figure~\ref{fig:sinh_profiles} shows $\phi$ for positive inputs.

\begin{figure}[H]
\centering
\definecolor{UBphi}{HTML}{2466A2}
\begin{minipage}[t]{0.65\linewidth}
\centering
\begin{tikzpicture}
\begin{axis}[
  width=0.97\linewidth,height=2.65in,
  xlabel={$t$},ylabel={$\phi(t)$},
  xmin=0,xmax=3,ymin=0,ymax=100,
  xtick={0,0.5,1,1.5,2,2.5,3},ytick={0,20,40,60,80,100},
  tick label style={font=\footnotesize},label style={font=\small},
  axis line style={black!55},tick style={black!55},
  grid=major,major grid style={black!10},
  minor tick num=0,
  clip=true,
]
\addplot[UBphi,line width=1.1pt] coordinates {
(0.25000000,239.982639768)
(0.26375000,215.609257633)
(0.27750000,194.767994998)
(0.29125000,176.807841573)
(0.30500000,161.22113924)
(0.31875000,147.607305918)
(0.33250000,135.64683451)
(0.34625000,125.082370514)
(0.36000000,115.704748577)
(0.37375000,107.342556497)
(0.38750000,99.8542438455)
(0.40125000,93.1220900493)
(0.41500000,87.04754754)
(0.42875000,81.5476130473)
(0.44250000,76.5519755696)
(0.45625000,72.000756683)
(0.47000000,67.8427066531)
(0.48375000,64.033754226)
(0.49750000,60.5358330093)
(0.51125000,57.3159257604)
(0.52500000,54.345281538)
(0.53875000,51.598770887)
(0.55250000,49.0543519252)
(0.56625000,46.6926260586)
(0.58000000,44.4964665312)
(0.59375000,42.4507064772)
(0.60750000,40.5418758272)
(0.62125000,38.7579785144)
(0.63500000,37.0883030787)
(0.64875000,35.5232610661)
(0.66250000,34.0542486589)
(0.67625000,32.6735277949)
(0.69000000,31.3741237023)
(0.70375000,30.1497363092)
(0.71750000,28.9946634214)
(0.73125000,27.9037339152)
(0.74500000,26.8722494807)
(0.75875000,25.8959336871)
(0.77250000,24.9708873367)
(0.78625000,24.0935492364)
(0.80000000,23.2606616477)
(0.81375000,22.4692397875)
(0.82750000,21.7165448481)
(0.84125000,21.0000600771)
(0.85500000,20.3174695289)
(0.86875000,19.6666391516)
(0.88250000,19.0455999202)
(0.89625000,18.4525327673)
(0.91000000,17.8857550955)
(0.92375000,17.3437086846)
(0.93750000,16.8249488311)
(0.95125000,16.328134579)
(0.96500000,15.8520199176)
(0.97875000,15.3954458389)
(0.99250000,14.95733316)
(1.00625000,14.5366760277)
(1.02000000,14.1325360308)
(1.03375000,13.7440368587)
(1.04750000,13.3703594455)
(1.06125000,13.0107375535)
(1.07500000,12.6644537486)
(1.08875000,12.3308357296)
(1.10250000,12.0092529754)
(1.11625000,11.6991136808)
(1.13000000,11.3998619501)
(1.14375000,11.1109752268)
(1.15750000,10.8319619347)
(1.17125000,10.5623593126)
(1.18500000,10.3017314234)
(1.19875000,10.0496673218)
(1.21250000,9.80577936727)
(1.22625000,9.56970166833)
(1.24000000,9.34108864673)
(1.25375000,9.11961371152)
(1.26750000,8.90496803305)
(1.28125000,8.69685940859)
(1.29500000,8.4950112117)
(1.30875000,8.29916141849)
(1.32250000,8.10906170424)
(1.33625000,7.92447660471)
(1.35000000,7.74518273693)
(1.36375000,7.5709680745)
(1.37750000,7.40163127333)
(1.39125000,7.2369810436)
(1.40500000,7.07683556452)
(1.41875000,6.92102193842)
(1.43250000,6.76937568134)
(1.44625000,6.62174024715)
(1.46000000,6.47796658283)
(1.47375000,6.33791271257)
(1.48750000,6.20144334845)
(1.50125000,6.06842952594)
(1.51500000,5.9387482622)
(1.52875000,5.81228223581)
(1.54250000,5.6889194861)
(1.55625000,5.56855313093)
(1.57000000,5.4510811016)
(1.58375000,5.33640589349)
(1.59750000,5.2244343317)
(1.61125000,5.11507735029)
(1.62500000,5.00824978452)
(1.63875000,4.90387017503)
(1.65250000,4.80186058316)
(1.66625000,4.70214641683)
(1.68000000,4.6046562661)
(1.69375000,4.50932174791)
(1.70750000,4.4160773594)
(1.72125000,4.32486033913)
(1.73500000,4.2356105359)
(1.74875000,4.14827028455)
(1.76250000,4.06278428832)
(1.77625000,3.97909950749)
(1.79000000,3.89716505365)
(1.80375000,3.81693208961)
(1.81750000,3.73835373432)
(1.83125000,3.6613849726)
(1.84500000,3.58598256947)
(1.85875000,3.51210498868)
(1.87250000,3.43971231532)
(1.88625000,3.36876618213)
(1.90000000,3.2992296995)
(1.91375000,3.23106738873)
(1.92750000,3.16424511855)
(1.94125000,3.09873004456)
(1.95500000,3.03449055165)
(1.96875000,2.97149619894)
(1.98250000,2.90971766742)
(1.99625000,2.84912670986)
(2.01000000,2.78969610312)
(2.02375000,2.73139960258)
(2.03750000,2.67421189855)
(2.05125000,2.61810857479)
(2.06500000,2.56306606873)
(2.07875000,2.50906163353)
(2.09250000,2.45607330182)
(2.10625000,2.40407985097)
(2.12000000,2.35306076997)
(2.13375000,2.30299622768)
(2.14750000,2.2538670425)
(2.16125000,2.20565465334)
(2.17500000,2.15834109188)
(2.18875000,2.11190895599)
(2.20250000,2.06634138433)
(2.21625000,2.02162203203)
(2.23000000,1.97773504744)
(2.24375000,1.93466504986)
(2.25750000,1.89239710828)
(2.27125000,1.85091672103)
(2.28500000,1.81020979625)
(2.29875000,1.77026263335)
(2.31250000,1.73106190516)
(2.32625000,1.69259464093)
(2.34000000,1.65484821008)
(2.35375000,1.61781030664)
(2.36750000,1.58146893447)
(2.38125000,1.54581239302)
(2.39500000,1.5108292639)
(2.40875000,1.47650839791)
(2.42250000,1.44283890279)
(2.43625000,1.40981013148)
(2.45000000,1.37741167095)
(2.46375000,1.34563333156)
(2.47750000,1.31446513697)
(2.49125000,1.28389731447)
(2.50500000,1.25392028587)
(2.51875000,1.22452465877)
(2.53250000,1.19570121833)
(2.54625000,1.16744091946)
(2.56000000,1.13973487934)
(2.57375000,1.11257437047)
(2.58750000,1.08595081393)
(2.60125000,1.0598557732)
(2.61500000,1.03428094811)
(2.62875000,1.0092181693)
(2.64250000,0.984659392929)
(2.65625000,0.960596695677)
(2.67000000,0.937022270075)
(2.68375000,0.913928420101)
(2.69750000,0.891307557055)
(2.71125000,0.869152195684)
(2.72500000,0.847454950562)
(2.73875000,0.826208532701)
(2.75250000,0.805405746392)
(2.76625000,0.785039486256)
(2.78000000,0.765102734506)
(2.79375000,0.745588558401)
(2.80750000,0.72649010788)
(2.82125000,0.707800613384)
(2.83500000,0.689513383833)
(2.84875000,0.671621804768)
(2.86250000,0.654119336641)
(2.87625000,0.636999513242)
(2.89000000,0.620255940267)
(2.90375000,0.603882294005)
(2.91750000,0.587872320144)
(2.93125000,0.572219832689)
(2.94500000,0.556918712984)
(2.95875000,0.541962908824)
(2.97250000,0.52734643367)
(2.98625000,0.513063365932)
(3.00000000,0.499107848344)
};
\addplot[black!45,densely dashed] coordinates {(0,0.5) (2.9991113435,0.5) (2.9991113435,0)};
\addplot[only marks,mark=*,mark size=2pt,UBphi] coordinates {(2.9991113435,0.5)};
\node[anchor=south east,font=\footnotesize,text=UBphi,inner sep=3pt]
  at (axis cs:2.9991113435,0.65) {$\phi(c)=\tfrac12$};
\end{axis}
\end{tikzpicture}
\end{minipage}
\caption{The potential $\phi$ on the input range $[0,3]$. The curve is clipped at the top: $\phi(t)\sim15/t^2$ as $t\to 0$.
The marked threshold is $c=\sqrt{3\operatorname{arsinh}(10)}\approx2.99911$, where $\phi(c)=1/2$.}
\label{fig:sinh_profiles}
\end{figure}
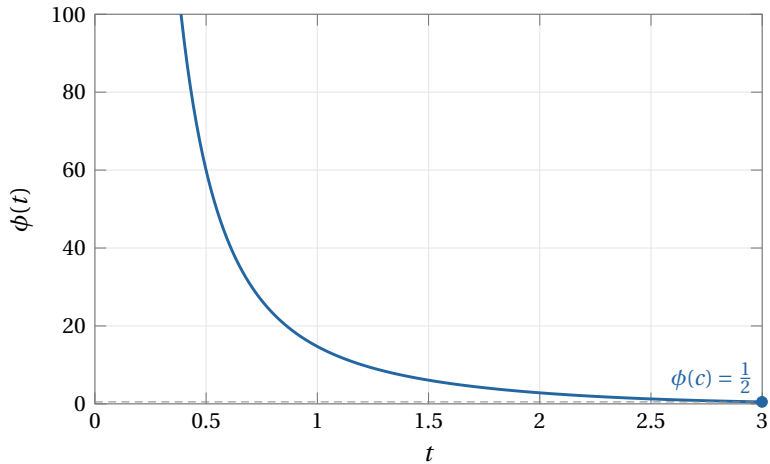

Set $s(z):=R(z)-4$ and define the compact minimization domain
\begin{equation}
\pazocal K:=\{z\in[-1,1]^n:R(z)\geq4,\ \langle a_i,z\rangle\leq\Delta\text{ for every }i\in[2d]\}.
\label{eq:sinh_domain}
\end{equation}
Define a potential on $\pazocal K$ as
\begin{equation}
\Phi(z):=\begin{cases}
\displaystyle\frac1{s(z)}\sum_{i=1}^{2d}\phi\Bigl(\frac{\Delta-\langle a_i,z\rangle}{\sqrt{s(z)}}\Bigr),&s(z)>0,\\[6pt]
0,&s(z)=0.
\end{cases}
\label{eq:sinh_potential}
\end{equation}
Proposition~\ref{prop:sinh_fractional} proves that $\Phi$ attains a minimum on $\pazocal K$ and that every minimizer has $R=4$.

\paragraph{Conceptual pseudocode.} For simplicity, we do not provide an algorithm to minimize $\Phi$, so our result is existential instead of polynomial time; we believe an efficient implementation is possible.
\begin{algorithm}[htbp]
\caption{Reciprocal-sinh minimization with nearest-sign rounding}
\label{alg:direct_potential_signing}
\begin{algorithmic}[1]
\Procedure{SignMatrix}{$A\in[-1,1]^{d\times n}$}
\State \textbf{if} $n\leq4$ \textbf{then} \Return $(1,\ldots,1)$
\State Define $\Delta:=\sqrt{3n\operatorname{arsinh}(10d/n)}$
\State Append the negation of every row to $A$; denote its rows by $a_1,\ldots,a_{2d}$
\State Define $\phi(t):=5/\sinh(t^2/3)$ for $t>0$ and $\phi(0):=\infty$
\State Define $R(z):=n-\|z\|_2^2$ and $s(z):=R(z)-4$
\State Define $\pazocal K:=\{z\in[-1,1]^n:R(z)\geq4,\ \langle a_i,z\rangle\leq\Delta\text{ for every }i\in[2d]\}$
\State Define, for $z\in\pazocal K$,
\Statex $\displaystyle\hspace{\algorithmicindent}\Phi(z):=\begin{cases}
\displaystyle\frac1{s(z)}\sum_{i=1}^{2d}\phi\Bigl(\frac{\Delta-\langle a_i,z\rangle}{\sqrt{s(z)}}\Bigr),&s(z)>0,\\[6pt]
0,&s(z)=0.
\end{cases}$
\State Define $z^*:=\operatorname*{arg\,min}_{z\in\pazocal K}\Phi(z)$ (choose any minimizer)
\For{$j=1,\ldots,n$}
\State Define $x_j:=1$ if $z_j^*\geq0$, and $-1$ otherwise
\EndFor
\State \Return $x$
\EndProcedure
\end{algorithmic}
\end{algorithm}

\label{sec:direct_potential_preliminaries}

We first state the scalar properties and then prove the two geometric ingredients: a capped trace bound and a second derivative formula.

\subsection{Scalar inequalities}
\label{sec:sinh_scalar_input}

For $t>0$, define
\[
\psi(t):=-t\phi'(t)-2\phi(t).
\]

The following lemma records the scalar properties of $\phi$ and $\psi$ needed. The inequality in part~(c) will be used in the dimension count, while the monotonicity in part~(d) bounds the potential at the origin.

\begin{lemma}[Scalar estimates]
\label{lem:sinh_scalar}
On positive arguments, the functions $\phi$ and $\psi$ have the following properties:
\begin{itemize}
\item[(a)] For every $t>0$, $\phi(t)>0$, $\phi'(t)<0$, and $\phi''(t)>0$. Also, $\phi(c)=1/2$ and $\phi(t)\longrightarrow\infty$ as $t\downarrow0$.
\item[(b)] For every $t>0$, $\psi(t)>0$.
\item[(c)] For every $t,\theta>0$, $\min\{\phi''(t),\theta\}\leq\frac34\psi(t)+\frac14\theta\phi(t)$.
\item[(d)] For every fixed $\Delta>0$, the function $s\mapsto s^{-1}\phi\Bigl(\frac{\Delta}{\sqrt s}\Bigr)$ is strictly increasing on $(0,\infty)$.
\end{itemize}
\end{lemma}

These are the only scalar facts used in the proof. Appendix~\ref{sec:sinh_scalar_reductions} gives the elementary calculations and the reduction of part~(c) to one scalar comparison.

\subsection{Capped trace bound}

The next lemma converts a lower bound on a weighted quadratic form into an upper bound on the dimension of the underlying subspace. Capped weights will allow us to apply Lemma~\ref{lem:sinh_scalar}(c) later.

\begin{lemma}[Capped trace]
\label{lem:sinh_capped_trace}
Let $a_1,\ldots,a_m\in[-1,1]^n$ be the rows of a matrix $A$ and let $V \subseteq \mathbb{R}^n$ be a subspace supported on $J\subseteq\{1,\ldots,n\}$. Suppose $w_1,\ldots, w_m\geq0$ and $L > 0$ satisfy
\begin{equation}
\sum_{i=1}^m w_i\langle a_i,v\rangle^2\geq L\|v\|_2^2\qquad \forall v\in V.
\label{eq:sinh_lower_curvature}
\end{equation}
Then
\begin{equation*}
\dim V  \le \frac{|J|}{L} \sum_{i=1}^m\min\Big\{w_i,\frac{L}{|J|}\Big\}.
\end{equation*}
\end{lemma}

\begin{proof}
Let $H:=\Big\{i \in [m] :w_i\geq \frac{L}{|J|}\Big\}$ and
\begin{equation*}
W:=\{v\in V:\langle a_i,v\rangle=0\text{ for every }i\in H\}.
\end{equation*}
Each orthogonality equation removes at most one dimension, so $k:=\dim W\geq\dim V- |H|$. Choose an orthonormal basis $u_1,\ldots,u_k$ of $W$. We have
\begin{equation}
kL\leq\sum_{i\notin H}w_i\sum_{\ell=1}^k\langle a_i,u_\ell\rangle^2
\leq\sum_{i\notin H}w_i\sum_{j\in J}A_{ij}^2
\leq |J| \sum_{i\notin H}w_i,
\label{eq:sinh_trace_sum}
\end{equation}
where the first step sums Eq.~\eqref{eq:sinh_lower_curvature} over the basis and uses orthogonality to the rows in $H$, the second step uses that $W \subseteq V$ is supported on $J$, and the third step uses $|A_{ij}|\leq1$. It follows that
\begin{equation*}
\sum_{i=1}^m\min\Big\{w_i,\frac{L}{|J|}\Big\} 
=\frac{|H| \cdot L}{|J|} + \sum_{i\notin H}w_i
\geq\frac{(|H|+k)L}{|J|} \geq\frac{L}{|J|} \dim V,
\end{equation*}
where the first inequality uses Eq.~\eqref{eq:sinh_trace_sum} and the last uses $|H|+k\geq\dim V$. Rearranging gives the lemma.
\end{proof}

\subsection{Second derivative formula}

Define
\begin{equation*}
\pazocal U:=\{z\in\R^n:s(z)>0,\ \langle a_i,z\rangle<\Delta\text{ for every }i\in[m]\}.
\end{equation*}

We compute the second derivative of the potential along directions orthogonal to $z$. At a minimizer, this derivative is nonnegative, giving the quadratic inequality required by Lemma~\ref{lem:sinh_capped_trace}.

\begin{lemma}
\label{lem:sinh_hessian}
Let $\Phi$ be given by the smooth formula in Eq.~\eqref{eq:sinh_potential}, with the sum taken over any rows $a_1,\ldots,a_m$ and $\Delta>0$. Fix $z\in\pazocal U$ and write $s:=s(z)= n-4-\|z\|_2^2$. For every $v\perp z$,
\begin{equation}
\begin{aligned}
s^2\frac{\partial^2}{\partial\lambda^2}\Phi(z+\lambda v)\big|_{\lambda=0}
={}&\sum_{i=1}^m\phi''\Bigl(\frac{\Delta-\langle a_i,z\rangle}{\sqrt s}\Bigr)\langle a_i,v\rangle^2-\sum_{i=1}^m\psi\Bigl(\frac{\Delta-\langle a_i,z\rangle}{\sqrt s}\Bigr)\|v\|_2^2.
\end{aligned}
\label{eq:sinh_hessian}
\end{equation}
\end{lemma}

\begin{proof}
{\bf Change in the residual scale.}
By the definition of $s$, we have
\begin{equation*}
s(z+\lambda v)=s-2\lambda\langle z,v\rangle-\lambda^2\|v\|_2^2=s-\lambda^2\|v\|_2^2.
\end{equation*}
Taking derivatives,
\begin{equation}
\frac{\partial}{\partial\lambda}\frac1{s(z+\lambda v)}\Big|_{\lambda=0}=0,
\qquad\frac{\partial^2}{\partial\lambda^2}\frac1{s(z+\lambda v)}\Big|_{\lambda=0}=\frac{2\|v\|_2^2}{s^2}.
\label{eq:sinh_inverse_scale}
\end{equation}

{\bf Change in the normalized gap.}
Along the same line,
\begin{equation}
\frac{\Delta-\langle a_i,z+\lambda v\rangle}{\sqrt{s(z+\lambda v)}}
=\frac{\Delta-\langle a_i,z\rangle-\lambda\langle a_i,v\rangle}{\sqrt{s-\lambda^2\|v\|_2^2}}.
\label{eq:sinh_normalized_path}
\end{equation}
Thus
\begin{align}
\frac{\partial}{\partial\lambda}\Bigl(\frac{\Delta-\langle a_i,z+\lambda v\rangle}{\sqrt{s(z+\lambda v)}}\Bigr)\Big|_{\lambda=0}
&=-\frac{\langle a_i,v\rangle}{\sqrt s},\label{eq:sinh_normalized_first}\\
\frac{\partial^2}{\partial\lambda^2}\Bigl(\frac{\Delta-\langle a_i,z+\lambda v\rangle}{\sqrt{s(z+\lambda v)}}\Bigr)\Big|_{\lambda=0}
&=\frac{(\Delta-\langle a_i,z\rangle)\|v\|_2^2}{s^{3/2}}.\label{eq:sinh_normalized_second}
\end{align}
Here the derivative of $(s-\lambda^2\|v\|_2^2)^{-1/2}$ vanishes at zero and its second derivative there is $\|v\|_2^2/s^{3/2}$. The numerator is affine in $\lambda$, so its second derivative is zero; the mixed product term vanishes as well.

{\bf Combining the derivatives.}
Define $t_i:=(\Delta-\langle a_i,z\rangle)/\sqrt s$. The product and chain rules give
\begin{align*}
\frac{\partial^2}{\partial\lambda^2}
\left[\frac1{s(z+\lambda v)}\phi\Bigl(\frac{\Delta-\langle a_i,z+\lambda v\rangle}{\sqrt{s(z+\lambda v)}}\Bigr)\right]\Big|_{\lambda=0}&=\frac{2\|v\|_2^2}{s^2}\phi(t_i) +\frac{\langle a_i,v\rangle^2}{s^2}\phi''(t_i) +\frac{t_i\|v\|_2^2}{s^2}\phi'(t_i)\\ &=\frac1{s^2}\left[\phi''(t_i)\langle a_i,v\rangle^2
-\psi(t_i)\|v\|_2^2\right].
\end{align*}
The first equality uses Eqs.~\eqref{eq:sinh_inverse_scale}--\eqref{eq:sinh_normalized_second}; the mixed product term is zero by Eq.~\eqref{eq:sinh_inverse_scale}. The last equality uses $2\phi(t_i)+t_i\phi'(t_i)=-\psi(t_i)$. Summing over all rows and multiplying by $s^2$ proves Eq.~\eqref{eq:sinh_hessian}.
\end{proof}

\subsection{Proof of Theorem~\ref{thm:upper_intro}}
\label{sec:direct_potential_proof}

We now combine the previous lemmas to show that every minimizer of $\Phi$ has $R(z) = 4$. 

\begin{proposition}
\label{prop:sinh_fractional}
Suppose $n>4$. The potential $\Phi$ attains a finite minimum on $\pazocal K$. Every minimizer $z$ satisfies $\|Az\|_\infty\leq\Delta$ and $R(z)=4$.
\end{proposition}

\begin{proof}
The origin lies in $\pazocal K\cap\pazocal U$, with $s(0)=n-4>0$. Lemma~\ref{lem:sinh_scalar}(d) gives
\begin{equation}
\Phi(0)=\frac{2d}{n-4}\phi\Bigl(\frac{\Delta}{\sqrt{n-4}}\Bigr)
<\frac{2d}{n}\phi\Bigl(\frac{\Delta}{\sqrt n}\Bigr)=1.
\label{eq:sinh_origin}
\end{equation}

The set $\pazocal K$ is nonempty and compact: it is a closed subset of the cube, since $R$ and the row inner products are continuous. We verify lower semicontinuity of $\Phi$ relative to $\pazocal K$. At a point with $s(z)>0$ and all row gaps positive, the formula is smooth. At a point with $s(z)>0$ and a zero row gap, the corresponding term tends to $\infty$ from positive gaps and is infinite at zero gaps, while the scale stays bounded away from zero. At a point with $s(z)=0$, the assigned value $0$ is at most the lower limit of nearby values, since $\Phi\geq0$. These cases cover $\pazocal K$, so $\Phi$ attains a minimum there. This minimum is finite by Eq.~\eqref{eq:sinh_origin}. Let $z$ be any minimizer. Then
\begin{equation}
\Phi(z)=\min_{\pazocal K}\Phi\leq\Phi(0)<1.
\label{eq:sinh_value}
\end{equation}
Because $z\in\pazocal K$, the row constraints and symmetry give $\|Az\|_\infty\leq\Delta$. If $R(z)=4$, both conclusions follow. Suppose instead that $R(z)>4$, so $s(z)>0$. Since $\Phi(z)<1$, $\|Az\|_\infty < \Delta$. Therefore $z\in\pazocal U$.

Define
\begin{equation*}
J:=\{j \in [n]:|z_j|<1\},\qquad
V:=\{v\in\R^n:\langle z,v\rangle=0, v_j=0 \ \forall j\notin J\}.
\end{equation*}
Then
\begin{equation}
|J| \ge \sum_{j\in J}(1-z_j^2) = R(z) > 4.
\label{eq:sinh_free_coordinates}
\end{equation}
Since $V$ is supported on $J$ with only one linear constraint $\langle z,v \rangle = 0$, we have $\dim V\geq|J|-1$.

For every fixed $v\in V$, both signs of sufficiently small $\lambda$ keep $z+\lambda v$ in $\pazocal K$: frozen coordinates stay fixed, free coordinates remain in $(-1,1)$, $R(z+\lambda v)=R(z)-\lambda^2\|v\|_2^2>4$ for small $\lambda$, and $\|A(z+\lambda v)\|_\infty < \Delta$ by continuity. The potential is therefore smooth along this line. By minimality of $z$, the second derivative of $\lambda \mapsto \Phi(z+\lambda v)$ at $\lambda = 0$ is nonnegative. Define
\[
w_i:=\phi''\Bigl(\frac{\Delta-\langle a_i,z\rangle}{\sqrt{s(z)}}\Bigr)>0,
\qquad L:=\sum_{i=1}^{2d}\psi\Bigl(\frac{\Delta-\langle a_i,z\rangle}{\sqrt{s(z)}}\Bigr)>0;
\]
positivity follows from Lemma~\ref{lem:sinh_scalar}(a) and~(b). Lemma~\ref{lem:sinh_hessian} yields $\sum_{i=1}^{2d}w_i\langle a_i,v\rangle^2\geq L\|v\|_2^2$ and so
\[
	\dim V  \le \frac{|J|}{L} \sum_{i=1}^{2d}\min\Big\{w_i,\frac{L}{|J|}\Big\}.
\]
By Eq.~\eqref{eq:sinh_value} and Eq.~\eqref{eq:sinh_free_coordinates}, we have
\begin{equation}
\sum_{i=1}^{2d}\phi\Bigl(\frac{\Delta-\langle a_i,z\rangle}{\sqrt{s(z)}}\Bigr)=s(z)\Phi(z)<s(z)\leq|J|-4.
\label{eq:sinh_budget}
\end{equation}
Therefore,
\begin{align*}
|J|-1\leq\dim V
&\leq\frac{|J|}{L} \sum_{i=1}^{2d}\min\Big\{w_i,\frac{L}{|J|}\Big\}\\
&\leq\frac{3|J|}{4L}\sum_{i=1}^{2d}\psi\Bigl(\frac{\Delta-\langle a_i,z\rangle}{\sqrt{s(z)}}\Bigr)
+\frac14\sum_{i=1}^{2d}\phi\Bigl(\frac{\Delta-\langle a_i,z\rangle}{\sqrt{s(z)}}\Bigr)\\
&=\frac{3|J|}4+\frac14\sum_{i=1}^{2d}\phi\Bigl(\frac{\Delta-\langle a_i,z\rangle}{\sqrt{s(z)}}\Bigr)\\
&<\frac{3|J|+|J|-4}4=|J|-1.
\end{align*}
The third inequality is Lemma~\ref{lem:sinh_scalar}(c), with $\theta=L/|J|$; the fourth is Eq.~\eqref{eq:sinh_budget}. This contradiction rules out $R(z)>4$. Therefore $R(z)=4$.
\end{proof}

\begin{proof}[Proof of Theorem~\ref{thm:upper_intro}]
For $n\leq4$, any signing has discrepancy at most $n\leq4<\Delta+4$. For $n>4$, Proposition~\ref{prop:sinh_fractional} yields $z\in [-1,1]^n$ with $\|Az\|_\infty\leq \Delta$ and $R(z)=4$. Rounding to the nearest sign vector $x$ gives
\begin{equation*}
\|Ax\|_\infty\leq\|Az\|_\infty+R(z)\leq \Delta+4=\sqrt{3n\operatorname{arsinh}(10d/n)}+4,
\end{equation*}
where the first inequality is Lemma~\ref{lem:sinh_rounding}.
\end{proof}
\section*{Acknowledgments}

This work was done while the second author was visiting Microsoft Research Redmond. The authors used GPT-5.5 Pro, GPT-5.6 Pro, GPT-6 Pro, Fable 5 and Fable 5.1 during this project to explore
proof strategies and assist with verification. Every AI-generated proof
was verified and rewritten by the authors, who take full responsibility
for the paper.
\FloatBarrier



\appendix


\section{Proof of Lemma~\ref{lem:binomial}}

\begin{proof} Throughout, let $\beta_k:=2^{-k}\binom{k}{k/2}$. We first prove Gautschi's inequality~\cite[(5.6.4)]{DLMF}
\begin{equation}\label{eq:gautschi}
\sqrt{\frac{2}{\pi(k+1)}}\le \beta_k\le
\sqrt{\frac{2}{\pi k}}.
\end{equation}
Let $I_r:=\int_0^{\pi/2}\sin^r x\,dx$. We have $I_r=(r-1)I_{r-2}/r$ for $r\ge2$, where the identity follows from integration by parts. The relevant values are
\[
I_k=\frac\pi2\,\beta_k,
\qquad
I_{k+1}=\frac1{(k+1)\beta_k},
\qquad
I_{k-1}=\frac1{k\beta_k}.
\]
These three identities follow by iterating the recurrence from $I_0=\pi/2$ and $I_1=1$. As $0\le\sin x\le1$ on $[0,\pi/2]$, $I_k$ is decreasing, so $I_{k+1}\le I_k\le I_{k-1}$. Thus the identities above give Eq.~\eqref{eq:gautschi}.

{\bf Proof of (a).} For the expectation identity, note that for $0\le r\le k$, the event $S_k=2r-k$ occurs precisely when $r$ of the $k$ signs are positive, so it has probability $2^{-k}\binom kr$. We obtain
\[
\begin{aligned}
\E[|S_k|]
&=2^{1-k}\sum_{r=k/2+1}^k(2r-k)\binom kr \\ & =k2^{1-k}\sum_{r=k/2+1}^k
\left[\binom{k-1}{r-1}-\binom{k-1}{r}\right]\\
&=k2^{1-k}\binom{k-1}{k/2}=k\beta_k
\le \sqrt{\frac2\pi}\,\sqrt k<0.8\sqrt k.
\end{aligned}
\]

\noindent{\bf Proof of (b).}
For every $s\in\R$, $\frac{e^s + e^{-s}}{2} =\sum_{r\ge0}\frac{s^{2r}}{(2r)!}
\le\sum_{r\ge0}\frac{s^{2r}}{2^r r!}=e^{s^2/2}$. Therefore,
\[
\E\left[\exp\left(\frac{tS_k}{\sqrt k}\right)\right]
=\left(\frac{\exp(\frac t{\sqrt k}) + \exp(-\frac t{\sqrt k})}{2}\right)^k
\le e^{t^2/2}.
\]
\medskip
\noindent{\bf Proof of (c).} Define \[
F_j:=2^{-k}\sum_{i=0}^j\binom ki
\qquad(0\le j\le k).
\]
We will use
\[
\frac{1-x}{1+x}\le e^{-2x}\qquad(0\le x<1),
\]
which follows by integrating
$\frac{d}{dx}\log\frac{1-x}{1+x}=-2/(1-x^2)\le-2$. We claim that
\begin{equation}\label{eq:Fbound}
F_j\le\frac12\exp\left(
\sqrt{\frac{2}{\pi k}}\,(2j-k+1)
\right)
\qquad(0\le j\le k).
\end{equation}
For $j\ge k/2$,
\[
F_j
\le\frac{1+\beta_k}{2}+(j-k/2)\beta_k
=\frac12\bigl(1+\beta_k(2j-k+1)\bigr)
\le\frac12\exp\left(
\sqrt{\frac{2}{\pi k}}\,(2j-k+1)
\right).
\]
Here the first step uses $F_{k/2}=(1+\beta_k)/2$ by symmetry and bounds each of the remaining $j-k/2$ point probabilities by the central probability $\beta_k$; the second step collects terms; and the third uses $1+y\le e^y$ together with the upper bound in Eq.~\eqref{eq:gautschi} and $2j-k+1\ge0$.

For the starting value below the center,
\begin{equation}\label{eq:binomial_sqrt_bound}
\sqrt{1+\frac1k}\le1+\frac1{2k}
\le1+\frac12\sqrt{\frac{2}{\pi k}}.
\end{equation}
The first step uses $\sqrt{1+x}\le1+x/2$ for $x\ge0$; and the second uses $1/k\le\sqrt{2/(\pi k)}$, which follows from $k\ge2$ and $\pi<4$. We obtain
\begin{equation}\label{eq:leftbase}
F_{k/2-1}=\frac{1-\beta_k}{2}
\le\frac12
\frac{1-\frac12\sqrt{2/(\pi k)}}{1+\frac12\sqrt{2/(\pi k)}}
\le\frac12\exp\left(-\sqrt{\frac{2}{\pi k}}\right).
\end{equation}
The first step follows from symmetry and the central probability $\beta_k$; the second combines the lower bound in Eq.~\eqref{eq:gautschi} with Eq.~\eqref{eq:binomial_sqrt_bound}, and then rearranges; and the third applies $(1-x)/(1+x)\le e^{-2x}$ with $x=\frac12\sqrt{2/(\pi k)}\in(0,1)$.

For $1\le j\le k/2-1$, we have
\begin{equation}\label{eq:binomial_ratio_product}
\frac{F_{j-1}}{2^{-k}\binom kj}
=\sum_{r=1}^j\prod_{q=0}^{r-1}
\frac{j-q}{k-j+q+1}.
\end{equation}
This identity follows by dividing the defining sum for $F_{j-1}$ by $2^{-k}\binom kj$, reindexing by $r:=j-i$, and expanding the binomial ratios as products. We next obtain
\begin{equation}\label{eq:binomial_ratio_intermediate}
\frac{F_{j-1}}{F_j}\le\frac{F_{k/2-2}}{F_{k/2-1}}
=1-\frac{2k\beta_k}{(k+2)(1-\beta_k)}
\le1-\frac{2k\sqrt{2/(\pi k)}}{(k+2)\bigl(\sqrt{1+1/k}-\sqrt{2/(\pi k)}\bigr)}.
\end{equation}
The first step follows because every factor in the product in Eq.~\eqref{eq:binomial_ratio_product} increases with $j$, and the sum gains a positive term when $j$ increases; hence $F_{j-1}/(F_j-F_{j-1})$, and therefore $F_{j-1}/F_j$, increases with $j$. The second step uses $F_{k/2-1}=(1-\beta_k)/2$ and $2^{-k}\binom{k}{k/2-1}=k\beta_k/(k+2)$, the latter following from the ratio of adjacent binomial coefficients. The third step uses the lower bound in Eq.~\eqref{eq:gautschi} and monotonicity of $b/(1-b)$ for $0<b<1$. Also,
\begin{equation}\label{eq:binomial_denominator_bound}
(k+2)\sqrt{1+\frac1k}
\le k+\frac52+\frac1k
\le k+3
\le k+2(k+1)\sqrt{\frac{2}{\pi k}}.
\end{equation}
The first step uses the first inequality in Eq.~\eqref{eq:binomial_sqrt_bound} and expands the product; the second uses $1/k\le1/2$; and the third uses $2(k+1)\sqrt{2/(\pi k)}\ge3$, which follows from $\sqrt{2/\pi}>3/4$ and $k+1\ge2\sqrt k$. Thus,
\begin{equation}\label{eq:Fratio}
\frac{F_{j-1}}{F_j}
\le\frac{1-\sqrt{2/(\pi k)}}{1+\sqrt{2/(\pi k)}}
\le\exp\left(-2\sqrt{\frac{2}{\pi k}}\right).
\end{equation}
The first step applies Eq.~\eqref{eq:binomial_denominator_bound} to Eq.~\eqref{eq:binomial_ratio_intermediate}, bounding the denominator by $k(1+\sqrt{2/(\pi k)})$ and simplifying. The second step applies $(1-x)/(1+x)\le e^{-2x}$ with $x=\sqrt{2/(\pi k)}\in(0,1)$. It follows that
\begin{equation}\label{eq:binomial_below_center}
F_{k/2-r}
\le\frac12\exp\left(-\sqrt{\frac{2}{\pi k}}\,(2r-1)\right)
\qquad(1\le r\le k/2),
\end{equation}
where the inequality starts from Eq.~\eqref{eq:leftbase} and applies Eq.~\eqref{eq:Fratio} $r-1$ times. Eq.~\eqref{eq:binomial_below_center} proves Eq.~\eqref{eq:Fbound} for $j<k/2$ by taking $r:=k/2-j$. When $k=2$, the ratio range is empty and \eqref{eq:leftbase} already suffices.

Finally, if $t\sqrt k-1<-k$, the desired probability is zero. Otherwise choose the largest $j\in\{0,\ldots,k\}$ with $2j-k\le t\sqrt k-1$; this also covers thresholds above the support. Then
\[
\Pp[S_k\le t\sqrt k-1]
=F_j
\le\frac12\exp\left(\sqrt{\frac2\pi}\,t\right),
\]
where the first step follows from the definition of $F_j$ and the choice of $j$; and the second uses Eq.~\eqref{eq:Fbound}, $2j-k+1\le t\sqrt k$, and monotonicity of the exponential.

\medskip
\noindent{\bf Proof of (d).} We consider three ranges of $u$. If $u\le1-\frac52\sqrt{\delta}$, the right-hand side already exceeds $1$.

Suppose next that $|u-1|\le\frac52\sqrt{\delta}$, and set $t:=\frac{1+\delta-u}{\sqrt\delta}+\frac1{\sqrt k}$. By Lemma~\ref{lem:binomial}(c),
\begin{equation}\label{eq:binomial_intermediate_tail}
\Pp\left[u+\sqrt{\frac\delta k}S_k\le1+\delta\right]
\le\frac12\exp\left(\sqrt{\frac2\pi}\,t\right).
\end{equation}
It remains to compute
\begin{equation}\label{eq:binomial_exponent_comparison}
\sqrt{\frac2\pi}\,t-\frac{1-u^2}{\sqrt{2\pi\delta}}
=\sqrt{\frac2\pi}\left(\frac{(1-u)^2+2\delta}{2\sqrt{\delta}}+\frac1{\sqrt k}\right)
\le\frac45\left(\frac{\frac{25}{4}\delta+2\delta}{2\sqrt{\delta}}+\frac{\sqrt{\delta}}2\right)
=\frac{37}{10}\sqrt{\delta}
<4\sqrt{\delta}.
\end{equation}
Combining Eq.~\eqref{eq:binomial_exponent_comparison} with Eq.~\eqref{eq:binomial_intermediate_tail} gives the claimed bound in this intermediate range.

Finally, suppose $u\ge1+\frac52\sqrt{\delta}$, and set $h:=(u-1)/\sqrt{\delta}\ge5/2$. By Markov's inequality and Lemma~\ref{lem:binomial}(b),
\[
\Pp\left[S_k\le-(h-\sqrt{\delta})\sqrt k\right]
=\Pp\left[e^{-(h-\sqrt{\delta})S_k/\sqrt k}\ge e^{(h-\sqrt{\delta})^2}\right]
\le e^{-(h-\sqrt{\delta})^2}\E\left[e^{-(h-\sqrt{\delta})S_k/\sqrt k}\right]\le e^{-(h-\sqrt{\delta})^2/2}.
\]
The first step exponentiates the event, using $h-\sqrt\delta>0$; the second applies Markov's inequality to $e^{-(h-\sqrt\delta)S_k/\sqrt k}$; and the third applies Lemma~\ref{lem:binomial}(b) with $t=-(h-\sqrt\delta)$ and combines the exponential factors.
In this range of $u$, $-(h-\sqrt{\delta})^2/2<\frac{1-u^2}{\sqrt{2\pi\delta}}-\log2$.
Thus $\Pp\left[u+\sqrt{\frac\delta k}S_k\le1+\delta\right]
\le\frac12\exp\left(\frac{1-u^2}{\sqrt{2\pi\delta}}\right).$
\end{proof}

\section{Proof of Lemma~\ref{lem:noise}}

\begin{proof}
{\bf Proof of (a).}
Let $s_i := 1$ for $z_i \ge 0$ and $-1$ otherwise. If $\|z+Y\|_\infty\le(1+\delta)\sqrt n$, then for every $i$,
\[
|z_i|+s_iY_i=s_i(z_i+Y_i)
\le|z_i+Y_i|
\le\|z+Y\|_\infty
\le(1+\delta)\sqrt n.
\]
By symmetry, the $s_iY_i$ are independent copies of $Y_1$. Since $\delta n\ge2^{24}$ is even, we may apply Lemma~\ref{lem:binomial}(d) with $u=|z_i|/\sqrt n$. Using $\|z\|_2^2=(1-\delta)n^2$, we obtain
\begin{align*}
\Pp\bigl[\|z+Y\|_\infty\le(1+\delta)\sqrt n\bigr]
&\le\prod_{i=1}^n
\Pp\left[\frac{|z_i|}{\sqrt n}+\frac{s_iY_i}{\sqrt n}\le1+\delta\right]\\
&\le2^{-n}\exp\left(
4\sqrt\delta\,n
+\frac{1}{\sqrt{2\pi\delta}}
\left(n-\frac{\|z\|_2^2}{n}\right)
\right)\\
&=2^{-n}\exp\left(
\left(4+\frac12\sqrt{\frac2\pi}\right)\sqrt\delta\,n
\right)\\
&\le2^{-n} 2^{n/320}.
\end{align*}

\medskip
\noindent{\bf Proof of (b).} Now suppose in addition $\|z\|_1\le0.99n^{3/2}$, and denote $I=\{i:|z_i|>\sqrt n\}$.
Then
\begin{align*}
\sum_{i\in I}(z_i^2-n)
&\ge\sum_{i=1}^n(z_i^2-\sqrt n\,|z_i|)\\
&=\|z\|_2^2-\sqrt n\,\|z\|_1\\
&\ge(0.01-\delta)n^2 \ge\frac{n^2}{200}.
\end{align*}
The first inequality uses $z_i^2-n\ge z_i^2-\sqrt n\,|z_i|$ on $I$ and $z_i^2-\sqrt n\,|z_i|\le0$ outside $I$. Applying Lemma~\ref{lem:binomial}(d) with $u=|z_i|/\sqrt n$ for $i\in I$, and using independence and symmetry as above, gives
\begin{align*}
\Pp\bigl[\|z+Y\|_\infty\le(1+\delta)\sqrt n\bigr]
&\le\prod_{i\in I}
\Pp\left[\frac{|z_i|}{\sqrt n}+\frac{s_iY_i}{\sqrt n}\le1+\delta\right]\\
&\le2^{-|I|}
\exp\left(
4\sqrt\delta\,|I|
-\frac{1}{n\sqrt{2\pi\delta}}
\sum_{i\in I}(z_i^2-n)
\right)\\
&\le\exp\left(
\left(4\sqrt\delta-
\frac1{400\sqrt{2\delta}}
\right)n
\right)\\
&\le2^{-3n}. \qedhere
\end{align*}
\end{proof}

\section{Scalar calculations for the upper bound}
\label{sec:sinh_scalar_reductions}

We give the elementary parts of Lemma~\ref{lem:sinh_scalar} and reduce its capped inequality to one scalar comparison.

\begin{proof}
{\bf Proof of (a) and (b).}
For $t>0$, define $u:=t^2/3$. Direct differentiation gives
\begin{align*}
\phi'(t)&=-\frac{10t\cosh u}{3\sinh^2u}<0,\\
\phi''(t)&=\frac{10}{3\sinh u}[-\coth u+2u(1+2/\sinh^2u)],\\
\psi(t)&=\frac{10}{\sinh u}(u\coth u-1).
\end{align*}
Here the hyperbolic cosine and hyperbolic cotangent are given by
\[
\cosh u:=\frac{e^u+e^{-u}}2,\qquad
\coth u:=\frac{\cosh u}{\sinh u}=\frac{e^u+e^{-u}}{e^u-e^{-u}}\qquad(u>0).
\]
The first inequality follows because $t>0$. The function $u\cosh u-\sinh u$ vanishes at zero and has derivative $u\sinh u>0$, which proves $\psi(t)>0$. After multiplication by $\sinh^2u$, the bracket in the expression for $\phi''$ becomes
\begin{equation*}
2u(\sinh^2u+2)-\sinh u\cosh u.
\end{equation*}
This expression vanishes at zero and has derivative $3+4u\sinh u\cosh u>0$, so $\phi''(t)>0$. Positivity of $\phi$, its divergence at zero, and $\phi(c)=1/2$ follow directly from its definition.

{\bf Proof of (c).}
We use the one-variable comparison
\begin{equation}
\Big(1-\frac{\phi(t)}4\Big)\phi''(t)\leq\frac34\psi(t)\qquad(t>0).
\label{eq:sinh_comparison}
\end{equation}
This is a scalar inequality for the positive-input formula $\phi(t)=5/\sinh(t^2/3)$, with $t>0$; its verification is omitted.

We consider two cases separately. One case is $\phi (t) \geq 4$. The other case is $\phi(t) \in (0,4)$.

In the first case, we can show
\begin{equation*}
\min\{\phi''(t),\theta\}\leq\theta\leq\frac14\theta\phi(t)\leq\frac34\psi(t)+\frac14\theta\phi(t),
\end{equation*}
where the last inequality follows from part~(b). 

For the second case when $0<\phi(t)<4$,
\begin{equation*}
\min\{\phi''(t),\theta\}\leq\Big(1-\frac{\phi(t)}4\Big)\phi''(t)+\frac{\phi(t)}4\theta
\leq\frac34\psi(t)+\frac14\theta\phi(t),
\end{equation*}
where the first inequality bounds the minimum by a convex combination, and the second step is Eq.~\eqref{eq:sinh_comparison}.

{\bf Proof of (d).}
Fix $\Delta>0$ and define $t:=\Delta/\sqrt s$. We have
\begin{equation}
\frac{\partial}{\partial s}\Big[\frac1s\phi\Bigl(\frac \Delta{\sqrt s}\Bigr)\Big]
=\frac{-\phi(t)-t\phi'(t)/2}{s^2}
=\frac{\psi(t)}{2s^2}>0,
\label{eq:sinh_scale_monotonicity}
\end{equation}
where the first step follows from the product and chain rules; the second step is the definition of $\psi$; and the third step follows from part~(b).
\end{proof}




\begin{thebibliography}{99}

\bibitem{AlonSpencer}
Noga Alon and Joel H. Spencer,
\emph{The Probabilistic Method}, 4th ed., Wiley, 2016, \S13.2,
\url{https://cs.nyu.edu/~spencer/sixsigma.pdf}.

\bibitem{Bansal}
Nikhil Bansal,
\emph{Constructive algorithms for discrepancy minimization},
Proceedings of FOCS 2010, 3--10,
\url{https://arxiv.org/abs/1002.2259}.

\bibitem{Belshaw}
Adrian William Belshaw,
\emph{Strong Normality, Modular Normality, and Flat Polynomials:
Applications of Probability in Number Theory and Analysis},
Ph.D. thesis, Simon Fraser University, 2013, Chapter~5,
\url{https://summit.sfu.ca/item/13657}.

\bibitem{LovettMeka}
Shachar Lovett and Raghu Meka,
\emph{Constructive discrepancy minimization by walking on the edges},
Proceedings of FOCS 2012, 489--498,
\url{https://arxiv.org/abs/1203.5747}.

\bibitem{PesentiVladu}
Lucas Pesenti and Adrian Vladu,
\emph{Discrepancy minimization via regularization},
Proceedings of SODA 2023, 1734--1758;
corrected version, arXiv:2211.05509v2, April 14, 2026,
\url{https://arxiv.org/abs/2211.05509v2}.

\bibitem{ZhaoNotes}
Yufei Zhao,
\emph{Probabilistic Methods in Combinatorics},
lecture notes, MIT 18.226, Fall 2020, \S5.1,
\url{https://yufeizhao.com/pm/fa20/probmethod_notes.pdf}.

\bibitem{BandeiraEtAl}
Afonso S. Bandeira, Anastasia Kireeva, Antoine Maillard, and Almut R\"odder,
\emph{Randomstrasse101: Open Problems of 2024},
arXiv:2504.20539 (2025),
\url{https://arxiv.org/abs/2504.20539}.

\bibitem{deLauney}
Warwick de Launey,
\emph{On the asymptotic existence of {Hadamard} matrices},
J. Combin. Theory Ser. A \textbf{116} (2009), 1002--1008,
\url{https://arxiv.org/abs/1003.4001}.

\bibitem{DLMF}
{National Institute of Standards and Technology},
\emph{{NIST Digital Library of Mathematical Functions}},
\S5.6(i), equation~(5.6.4), Gautschi's inequality,
\url{https://dlmf.nist.gov/5.6.E4}.

\bibitem{Hoeffding}
Wassily Hoeffding,
\emph{Probability inequalities for sums of bounded random variables},
J. Amer. Statist. Assoc. \textbf{58} (1963), no.~301, 13--30.

\bibitem{KavutYucel}
Sel\c{c}uk Kavut and Melek Diker Y\"ucel,
\emph{9-variable {Boolean} functions with nonlinearity 242 in the generalized
  rotation symmetric class},
Information and Computation \textbf{208} (2010), no.~4, 341--350,
\url{https://doi.org/10.1016/j.ic.2009.12.002}.

\bibitem{Massart}
Pascal Massart,
\emph{{Tusnady}'s lemma, 24 years later},
Ann. Inst. H. Poincar\'e Probab. Statist. \textbf{38} (2002), no.~6, 991--1007.

\bibitem{OlsonSpencer}
John E. Olson and Joel Spencer,
\emph{Balancing families of sets},
J. Combin. Theory Ser. A \textbf{25} (1978), no.~1, 29--37.

\bibitem{Paley}
R. E. A. C. Paley,
\emph{On orthogonal matrices},
Journal of Mathematics and Physics \textbf{12} (1933), 311--320,
\url{https://doi.org/10.1002/sapm1933121311}.

\bibitem{ReisThesis}
Victor Reis,
\emph{Vector Balancing and Integer Programming},
Ph.D. thesis, University of Washington, 2023,
\url{https://tinyurl.com/VBandIP}.

\bibitem{Rothvoss}
Thomas Rothvoss,
\emph{Probabilistic Combinatorics},
lecture notes, University of Washington, Winter 2019, version of March 15,
  2019.

\bibitem{Schmidt}
{Kai-Uwe} Schmidt,
\emph{Asymptotically optimal {Boolean} functions},
J. Combin. Theory Ser. A \textbf{164} (2019), 50--59,
\url{https://doi.org/10.1016/j.jcta.2018.12.005}.

\bibitem{Spencer}
Joel Spencer,
\emph{Six standard deviations suffice},
Trans. Amer. Math. Soc. \textbf{289} (1985), no.~2, 679--706,
\url{https://doi.org/10.1090/S0002-9947-1985-0784009-0}.

\bibitem{Spielman}
Daniel Spielman,
\emph{Discrepancy Theory and Randomized Controlled Trials},
Mathematical Picture Language Seminar, Harvard University, recorded talk,
  October 31, 2023, at 19:40,
\url{https://www.youtube.com/watch?v=KhUaImzb81w&t=1180s}.

\bibitem{Tao}
Terence Tao,
\emph{{Talagrand}'s concentration inequality},
What's new, June 9, 2009.

\end{thebibliography}
\end{document}